\documentclass[11pt,fleqn]{article}
\usepackage{amssymb}
\usepackage{amsfonts}
\usepackage{amsmath}
\usepackage{amsthm}
\usepackage{graphicx}
\usepackage{epsfig}
\usepackage{psfrag}
\usepackage{color}
\usepackage{dsfont}
\makeatletter
\xdef\@endgadget#1{{\unskip\nobreak\hfil\penalty50\hskip1em\hbox{}\nobreak
    \hfil#1\parfillskip=0pt\finalhyphendemerits=0\par}}
\def\@qedsymbol{${}_\blacksquare$}
\def\qed{\@endgadget{\@qedsymbol}}
\newtheorem{lemma}{Lemma}[section]

\newtheorem{proposition}[lemma]{Proposition}
\newtheorem{remark}[lemma]{Remark}
\newcommand{\mR}{\mathbb{R}}

\newcommand{\cL}{\mathcal{L}}
\newcommand{\cH}{\mathcal{H}}

\newcommand{\bq}{\begin{equation}}
\newcommand{\eq}{\end{equation}}
\newcommand{\bma}{\begin{bmatrix}}
\newcommand{\ema}{\end{bmatrix}}

\def\BibTeX{{\rm B\kern-.05em{\sc i\kern-.025em b}\kern-.08em
    T\kern-.1667em\lower.7ex\hbox{E}\kern-.125emX}}

\title{\LARGE \bf Suboptimal control by primal-dual \\gradient dynamics}

\author{Arjan van der Schaft
\thanks{A.J. van der Schaft is with the Bernoulli Institute for Mathematics, Computer
Science and AI, Jan C. Willems Center for Systems and Control, University of Groningen, PO Box 407, 9700 AK, the
Netherlands,
        {\tt\small A.J.van.der.Schaft@rug.nl}}
}

\begin{document}

\maketitle
\thispagestyle{empty}
\pagestyle{empty}

\begin{abstract}
This note generalizes the port-Hamiltonian formulation of the continuous time primal-dual gradient algorithm for static constrained convex optimization to the convex optimal control problem.The resulting dynamics is shown to be a port-Hamiltonian system of partial differential equations, involving ordinary physical time as well 'algorithmic' time. Convergence to the optimal control solution is indicated, and it is argued that sub-optimal control strategies could be derived starting from the partial differential equation formulation.
\end{abstract}


\section{Introduction}
The purpose of this note is to clarify the use of continuous time primal-dual gradient dynamics for obtaining sub-optimal control strategies. Inspired by \cite{gernandt}, the primal-dual gradient algorithm for convex optimal control is explicitly formulated as a system of partial differential equations. While it is difficult to claim full originality in this vast area, it seems that such a point of view has not yet been emphasized in the literature. Furthermore, the port-Hamiltonian perspective taken in this note sheds new light, especially when it comes to interconnection properties and stability analysis of the resulting dynamics. The note only contains basic observations and computations, which however could provide starting points for further mathematical and computational developments.

Point of departure is the use of primal-dual gradient dynamics in continuous time (PDGCT) for \emph{constrained convex optimization}; dating back at least to \cite{arrow}, with numerous subsequent developments. In \cite{stegink1, stegink2, passivitybook}, see also \cite{camlibel2}, it was realized that in the static case PDGCT defines an \emph{incremental port-Hamiltonian system}, as coined in \cite{camlibel1,camlibel2}. This was used for \emph{stability} analysis by using the shifted Hamiltonian as Lyapunov function, and for \emph{interconnection} purposes. In fact, in \cite{stegink1,stegink2} market dynamics modelled as PDGCT with additional inputs and outputs was coupled to the port-Hamiltonian formulation of power networks, leading to a closed-loop port-Hamiltonian dynamics for which the shifted total Hamiltonian is a Lyapunov function; thus proving overall stability. This point of view was subsequently extended to the infinite-dimensional case in \cite{gernandt}, see also \cite{preuster, lefevre}, introducing PDGCT as \emph{optimizer dynamics} to be coupled to other port-Hamiltonian systems. The present note aims at clarifying the precise connection of infinite-dimensional PDGCT with (sub-)optimal control. In particular, the equilibrium of infinite-dimensional PDGCT as described in \cite{gernandt} corresponds to the optimal control solution in case of linear dynamics with convex cost criterion. Furthermore, it is shown how infinite-dimensional PDGCT amounts to a system of partial differential equations (pde's), involving both 'physical' time (of the optimal control problem) and 'algorithmic' time (of the primal-dual gradient algorithm). This system of pde's exhibits two Hamiltonian structures. One is corresponding to Pontryagin's minimum principle, which defines a port-Hamiltonian system of differential-algebraic equations in the state, co-state, and input variables in physical time, see \cite{pHDAE}. The other is the port-Hamiltonian structure of infinite-dimensional PDGCT involving algorithmic time \cite{JGP}.

The explicit port-Hamiltonian pde formulation of PDGCT for optimal control as described in this paper has at least two potential applications. One is the interpretation of infinite-dimensional PDGCT as \emph{sub-optimal} control. Recently there has been much interest in rendering model predictive control (MPC) less computationally intensive by resorting to various sub-optimal versions; see \cite{rawlings} for a general MPC background. In \cite{yoshida} the use of PDGCT has been advocated for this purpose ('instant MPC'); however applying PDGCT to \emph{discrete time} MPC. Note in this respect that the port-Hamiltonian interpretation of PDGCT provides passivity guarantees; and thus there is no need for further dissipativity analysis as in \cite{yoshida}. Furthermore, the present paper suggests to start from the pde formulation of PDGCT for optimal control, and to derive computational schemes from there.
The other application concerns the interconnection to other port-Hamiltonian systems, either of physical or of optimization nature (distributed optimal control), and their stability analysis using shifted Hamiltonians. This point of view is already investigated in \cite{lefevre}, where control by interconnection of constrained port-Hamiltonian systems is approached by using PDGCT in order to handle state and input constraints, as well as in the recent papers \cite{gernandt, preuster} dealing with control by interconnection by 'optimizer dynamics'. 

\smallskip

The structure of this paper is as follows. Section 2 recalls the basics of PDGCT for static constrained convex optimization, and its (incremental) port-Hamiltonian formulation. Section 3 follows the same steps, but then within an infinite-dimensional optimal control context. First, it is discussed how convex optimal control, by interpreting the linear dynamics as constraints, leads to infinite-dimensional PDGCT. Then, it is shown how the resulting infinite-dimensional dynamics is given by a set of partial differential equations in incremental port-Hamiltonian form; thus explicating \cite{gernandt}. Using the shifted Hamiltonian with respect to the optimal control solution that is provided by Pontryagin's minimum principle convergence to the optimal control is indicated.

\section{Static constrained convex optimization}
In order to set the stage for the optimal control primal-dual gradient algorithm in continuous time in Section 3 we first recall the corresponding situation for the \emph{static} constrained convex optimization case. Consider the problem of minimizing a strictly convex and differentiable function $V:\mR^n \to \mR$ under affine constraints $Cq=b, q \in \mR^n,$ for a $k \times n$ matrix $C$ with rank $k$ and $k$-vector $b$. That is, consider the convex optimization problem
\bq
\min_{\{q \mid Cq=b\}} V(q).
\eq
Define the Lagrangian $L(q, \lambda) := V(q) + \lambda^\top (Cq-b)$, with $\lambda \in \mR^k$ a vector of Lagrange multipliers. Due to strict convexity of $V$ and full rank of $C$ there exists a unique pair $(q_*,\lambda_*)$ such that 
\bq
\label{eq:equilibriumstatic}
\frac{\partial L}{\partial q}(q_*, \lambda_*) =0, \; \frac{\partial L}{\partial \lambda}(q_*, \lambda_*) =0
\eq
(that is, $\frac{\partial V}{\partial q}(q_*) + C^\top \lambda_*=0,  Cq_*=b$),
and the solution to the minimization problem is given by $q_*$. 
A standard way to compute $(q_*,\lambda_*)$ is the \emph{primal-dual gradient algorithm}, where gradient \emph{descent} is taken with respect to $q$, and gradient \emph{ascent} with respect to the Lagrange multipliers $\lambda$. As is well-known, the primal-dual gradient algorithm is compromizing between minimization of $V$ and constraint satisfaction. The \emph{continuous time} version of this algorithm (PDGCT) is, cf. \cite{arrow},
\bq
\label{eq:primaldualstatic0}
\begin{array}{rcl}
\frac{d}{d \tau} Qq & = & - \frac{\partial L}{\partial q}(q, \lambda) \\[2mm]
\frac{d}{d \tau} \Lambda \lambda & = & \frac{\partial L}{\partial \lambda}(q, \lambda),
\end{array}
\eq
with $Q, \Lambda$ positive matrices determining the time-scales of the algorithm. 
As observed in \cite{stegink1,stegink2}, see also \cite{passivitybook}, PDGCT can be formulated as an incremental \emph{port-Hamiltonian system}. First define the new variables
$z:= Q q, \quad \mu:= \Lambda \lambda$.
In these so-called energy variables \eqref{eq:primaldualstatic0} takes the form
\bq
\label{eq:primaldualstatic}
\frac{d}{d \tau} \bma z \\[2mm] \mu \ema = \bma 0 & -C^\top \\[2mm] C & 0 \ema \bma q \\[2mm] \lambda \ema - \bma \frac{\partial V}{\partial q}(q) \\[2mm] b \ema, 
\eq
with $q= Q^{-1} z, \lambda= \Lambda^{-1} \mu$.
Then by defining the \emph{Hamiltonian} 
\bq
\label{eq:Hamiltonianstatic}
H(z,\mu):= \frac{1}{2} z^\top Q^{-1} z + \frac{1}{2} \mu^\top \Lambda^{-1} \mu,
\eq
and noting that $q=\frac{\partial H}{\partial z}(z,\mu), \lambda=\frac{\partial H}{\partial \mu}(z,\mu)$ are thus co-energy variables, it follows that \eqref{eq:primaldualstatic} is an incremental port-Hamiltonian system as coined and analyzed in \cite{camlibel1, camlibel2}. Indeed the mapping
\bq
\bma q \\[2mm] \lambda \ema \mapsto - \bma 0 & -C^\top \\[2mm] C & 0 \ema \bma q \\[2mm] \lambda \ema + \bma \frac{\partial V}{\partial q}(q) \\[2mm] b \ema
\eq
is a monotone mapping.

As a next step, define the \emph{shifted} Hamiltonian
\bq
\label{eq:Hamiltonianstaticshifted}
\widetilde{H}(z,\mu) = \frac{1}{2} (z - z_*)^\top Q^{-1} (z -z_*) + \frac{1}{2} (\mu - \mu_*)^\top \Lambda^{-1} (\mu - \mu_*),
\eq
where $z_*= Q  q_*, \mu_* = \Lambda \lambda_*,$ with $q_*,\lambda_*$ the unique solution of \eqref{eq:equilibriumstatic}. Then \eqref{eq:primaldualstatic} can be rewritten as the port-Hamiltonian system \cite{jeltsema, passivitybook}
\bq
\label{eq:primaldualstaticshifted}
\frac{d}{d \tau} \bma z \\[2mm] \mu \ema = \bma 0 & -C^\top \\[2mm] C & 0 \ema  
%\bma q \\[2mm] \lambda \ema 
\bma \frac{\partial \widetilde{H}}{\partial z}(z,\mu) \\[2mm] \frac{\partial \widetilde{H}}{\partial \mu}(z,\mu) \ema
- \bma \frac{\partial V}{\partial q}(q) - \frac{\partial V}{\partial q}(q^*)\\[2mm] 0 \ema. 
%\quad q= \frac{\partial \widetilde{H}}{\partial z}(z,\mu), \lambda = \frac{\partial \widetilde{H}}{\partial \mu}(z,\mu).
\eq
In particular one computes, using skew-symmetry of the first matrix on the right-hand side of \eqref{eq:primaldualstaticshifted},
\bq
\label{eq:staticdissipation}
\begin{array}{rcl}
\frac{d}{d\tau} \widetilde{H} & = &-(z-z_*)^\top Q^{-1} [\frac{\partial V}{\partial q}(q) - \frac{\partial V}{\partial q}(q_*)] \\[2mm]
&= & -(q-q_*)^\top [\frac{\partial V}{\partial q}(q) - \frac{\partial V}{\partial q}(q_*)] \leq 0,
\end{array}
\eq
where the inequality follows from convexity of $V$, and thus monotonicity of the map $q \mapsto \frac{\partial V}{\partial q}(q)$. In fact, the shifted Hamiltonian $\widetilde{H}$ equals the Lyapunov function proposed in \cite{arrow}, and thus the addition to the existing literature resides primarily in the port-Hamiltonian interpretation. 

The inequality \eqref{eq:staticdissipation} together with LaSalle's invariance principle implies convergence of the primal-dual gradient algorithm to $(q_*,\lambda_*)$. This well-known fact (in its basic form; see e.g. \cite{cherukuri} for extensions) is recalled here for completeness and as a precursor to Section 3.
\begin{proposition}\label{prop}
$\widetilde{H}$ is a Lyapunov function for the port-Hamiltonian system \eqref{eq:primaldualstaticshifted}, and for any initial condition $z(0)=Qq(0), \mu(0)= \Lambda \lambda(0)$, the system converges to $z_*= Q  q_*, \mu_* = \Lambda \lambda_*$. Hence PDGCT \eqref{eq:primaldualstatic0} converges to $q_*, \lambda_*$.
\end{proposition}
{\it Proof}
Equation \eqref{eq:staticdissipation} together with strict convexity of $V$ implies $\{ (z,\mu) \mid \frac{d}{d \tau}{\widetilde{H}}=0 \} = \{(z,\mu) \mid z=z_* \}$. It is easily seen that the largest invariant set within this set is the single point $(z_*,\mu_*)$. Hence by LaSalle's invariance principle the port-Hamiltonian system \eqref{eq:primaldualstaticshifted} converges to $(z_*,\mu_*)$, and thus the primal-dual gradient algorithm \eqref{eq:primaldualstatic} to $(q_*,\lambda_*)$.
${}_\blacksquare$

\smallskip

The incremental port-Hamiltonian system formulation \eqref{eq:primaldualstatic} of the PDCCT algorithm can be augmented with conjugated \emph{inputs} $u$ and \emph{outputs} $y$, leading to 
\bq
\label{eq:primaldualstaticuy}
\begin{array}{rcl}
\frac{d}{d \tau} \bma z \\[2mm] \mu \ema & \!=\! & \bma 0 & -C^\top \\[2mm] C & 0 \ema \bma q \\[2mm] \lambda \\[2mm]  \ema  - \bma \frac{\partial V}{\partial q}(q) - \frac{\partial V}{\partial q}(q_*)\\[2mm] 0 \ema \\[3mm]
 && + \; G(z,\mu) u \\[2mm]
y & = & G^\top (z,\mu) \bma \frac{\partial \widetilde{H}}{\partial z}(z,\mu) \\[2mm] \frac{\partial \widetilde{H}}{\partial \mu}(z,\mu) \ema,
\end{array}
\eq
for a suitably dimensioned matrix $G(z,\mu)$. Using compositionality of incremental port-Hamiltonian systems \cite{camlibel2, jeltsema}, \eqref{eq:primaldualstaticuy} can be \emph{interconnected} to other port-Hamiltonian systems, resulting in an interconnected port-Hamiltonian system; see \cite{stegink1,stegink2} for an application. 
%In fact, this was the original scenario studied in \cite{stegink1,stegink2}, where market dynamics was coupled to the port-Hamiltonian model of a power network and stability was shown using the sum of the Hamiltonians. 
This can be also used for \emph{distributed} optimization \cite{camlibel2}.
\begin{remark}\label{rem}
{\rm
All this can be generalized in several ways. First of all, next to equality constraints $Cq=b$ one can also include convex \emph{inequality} constraints within the incremental port-Hamiltonian formulation \cite{stegink1, stegink2}. Furthermore, the use of weighting matrices $Q,\Lambda$ can be generalized by considering functions $M(q,\lambda)$ with positive definite Hessian, and then replacing the Hamiltonian \eqref{eq:Hamiltonianstatic} by the Legendre transform $H(z,\mu):=M^*(z, \mu)$, where $z=\frac{\partial M}{\partial q}(q,\lambda), \mu=\frac{\partial M}{\partial \lambda}(q,\lambda)$, and considering instead of \eqref{eq:Hamiltonianstaticshifted} the shifted version of $H$ with respect to $z_*,\mu_*$; see \cite{passivitybook, stegink1, jeltsema}. Another generalization is to drop the assumption of differentiability of the convex function $V$, by replacing gradients by subdifferentials; see e.g. \cite{camlibel2}.
}
\end{remark}

\section{The primal-dual gradient algorithm for optimal control}
Consider now the \emph{optimal control} problem of minimizing over the vector-valued function $u(t), t \in [0,T],$ a cost criterion of the form
\bq
\int_0^T K(x(t),u(t)) dt,
\eq
where $T>0$ and $K(x,u)$ is a strictly convex function,
for a linear system dynamics
\bq
\dot{x}=Ax + Bu, \; x(0)=x_0, \quad x\in \mR^n, u\in \mR^m.
\eq
As is well-known, a solution to this problem is provided by Pontryagin's Minimum principle. Define the \emph{optimal control Hamiltonian} $H_{\rm{opt}}(x,u,p):= p^\top \left(Ax + Bu\right) + K(x,u)$. Then the first-order optimality conditions for an optimal solution $(x_*(t),u_*(t), p_*(t)), t \in [0,T],$ are given as
\bq
\label{eq:pont}
\begin{array}{rcl}
\dot{x}_* & = & \frac{\partial H_{\rm{opt}}}{\partial p}(x_*,u_*,p_*) = Ax_* + Bu_*, \quad x_*(0)=x_0 \\[2mm]
0 & = & \frac{\partial H_{\rm{opt}}}{\partial u}(x_*,u_*,p_*) = B^\top p_* + \frac{\partial K}{\partial u}(x_*,u_*) \\[2mm]
\dot{p}_* & = & -\frac{\partial H_{\rm{opt}}}{\partial x}(x_*,u_*,p_*) \\[2mm]
 & = & -A^\top p_* - \frac{\partial K}{\partial x}(x_*,u_*),  \quad p_*(T)=0.
\end{array}
\eq
In fact, due to strict convexity of $K$ and linearity of the system dynamics, it follows from Mangasarian's sufficiency theorem, see e.g. \cite{meinsma}, that $u_*(t), t \in [0,T],$ is the optimal control.

Alternatively, see e.g. \cite{meinsma}, by interpreting the system dynamics $\dot{x}=Ax + Bu$ as \emph{constraints} $Ax(t) + Bu(t) - \dot{x}(t)=0, t \in [0,T]$, the equations \eqref{eq:pont} arise as Euler-Lagrange equations for the calculus of variations problem of minimizing (over the functions $x,p,u$) the integral $\int_0^T F(x(t),\dot{x}(t),p(t),u(t)) dt$, with $F$ given as
\bq
F(x,\dot{x},u,p):= p^\top (Ax + Bu - \dot{x}) + K(x,u).
\eq
Here the time-functions $p:[0,T] \to \mR^n$ are infinite-dimensional Lagrange multipliers corresponding to the constraints $Ax(t) + Bu(t) - \dot{x}(t)=0, t \in [0,T]$. In analogy with the static optimization case this leads to the following \emph{infinite-dimensional} primal-dual gradient algorithm in continuous time. Define the functional
\bq
\cL(x,u,p):= \int_0^T F(x(t),\dot{x}(t),u(t),p(t)) dt,
\eq
depending on the time-functions $x(t),u(t), p(t), t \in [0,T]$. Then the continuous time primal-dual gradient algorithm in the new time-variable $\tau \in [0,\infty]$ is given by the dynamics
\bq
\label{eq:primaldualinf}
\frac{\partial}{\partial \tau} \bma Xx \\[2mm] Uu \\[2mm] Pp \ema =
\bma - \frac{\delta \cL}{\delta x} \\[2mm] - \frac{\delta \cL}{\delta u} \\[2mm]  \frac{\delta \cL}{\delta p} \ema,
\eq
for positive definite weighting matrices $X,U,P$. 
Here $\frac{\delta \cL}{\delta x}, \frac{\delta \cL}{\delta u}, \frac{\delta \cL}{\delta p}$ denote the \emph{variational derivatives} of the functional $\cL$ with respect to the time-functions $x(t),u(t),p(t), t \in [0,T]$. Note that the equations \eqref{eq:primaldualinf} involve 'physical' time $t \in [0,T]$, as well as 'algorithmic' time $\tau \in [0,\infty)$. So the functions $x,u,p$ actually depend on both, i.e., $x(t,\tau),p(t,\tau),u(t, \tau), t\in [0,T], \tau \in [0,\infty)$. (But variational derivatives are taken with respect to the functions of time $t$; for \emph{any} algorithmic time $\tau$.)

Since no time-derivatives (with respect to $t$) of $u$ and $p$ appear in $\cL$, the variational derivatives $\frac{\delta \cL}{\delta u}, \frac{\delta \cL}{\delta p}$ are simply $\frac{\partial F}{\partial u}= B^\top p(t) + \frac{\partial K}{\partial u}(x(t),u(t))$, respectively $\frac{\partial F}{\partial p}=Ax(t) + Bu(t) -\dot{x}(t), t \in [0,T]$. On the other hand, the variational derivative of $\int_0^T-p^\top (t) \dot{x}(t) dt$ with respect to the function $x(t), t \in [0,T],$ is computed as $\dot{p}(t), t \in [0,T]$ (note the cancelling of the minus sign due to partial integration). Therefore $\frac{\delta \cL}{\delta x} = \dot{p}(t) + A^\top p(t) + \frac{\partial K}{\partial x}(x(t),u(t)), t \in [0,T]$. Thus the infinite-dimensional primal-dual gradient algorithm \eqref{eq:primaldualinf} takes the explicit form
\bq
\label{eq:pde}
\frac{\partial}{\partial \tau} \bma Xx \\[2mm] Uu \\[2mm] Pp \ema \! = \!
\bma 
0 & 0 & - \frac{\partial}{\partial t} - A^\top \\[2mm]
0 & 0 & -B^\top \\[2mm]
- \frac{\partial}{\partial t} + A & B & 0
\ema \!
\bma x \\[2mm] u \\[2mm] p \ema -
\bma \frac{\partial K}{\partial x}(x,u) \\[2mm] \frac{\partial K}{\partial u}(x,u) \\[2mm] 0 \ema
\eq
Here the differential operator 
\bq
\label{eq:J}
J:= \bma 
0 & 0 & - \frac{\partial}{\partial t} - A^\top \\[2mm]
0 & 0 & -B^\top \\[2mm]
- \frac{\partial}{\partial t} + A & B & 0
\ema,
\eq
is \emph{formally skew-adjoint}. Hence the right-hand of \eqref{eq:pde} involves the sum of a formally skew-adjoint operator and the cyclically monotone operator
\bq
\bma x \\u \\ p \ema \mapsto \bma \frac{\partial K}{\partial x}(x,u) \\[2mm] \frac{\partial K}{\partial u}(x,u) \\[2mm] 0 \ema,
\eq
together defining a \emph{monotone} operator. 

Now, similarly to the static case, introduce the new functions
\bq
z(t,\tau):= X x(t,\tau), \, v(t, \tau):=U u(t,\tau), \, r(t, \tau):= P p(t, \tau)
\eq
Furthermore, in analogy with \eqref{eq:Hamiltonianstatic}, define for every $\tau \in [0,\infty)$ the \emph{Hamiltonian functional}
\bq
\label{eq:Hamiltonianinf}
\begin{array}{l}
\cH(z,v,r) :=  \frac{1}{2} \int_0^T  \left[z^\top (t,\tau) X^{-1} z(t,\tau) \, + \right.\\[2mm]
  \qquad  \left. v^\top (t,\tau) U^{-1}v (t,\tau) + r(t,\tau)^\top P^{-1} r(t,\tau) \right] dt
\end{array}
\eq
Then, since 
%\bq
%\begin{array}{rcl}
$\frac{\delta \cH}{\delta z} = X^{-1} z(t,\tau)=x(t,\tau),$$ \frac{\delta \cH}{\delta v} = U^{-1} v(t,\tau)=u(t,\tau), 
$$ \frac{\delta \cH}{\delta r} = P^{-1} r(t,\tau)=p(t,\tau),$
%\end{array}
%\eq
it follows that \eqref{eq:pde} defines an \emph{infinite-dimensional incremental port-Hamiltonian system} \cite{camlibel1,camlibel2}. 
\begin{remark}
{\rm
Refer to \cite{gernandt} for specification of suitable domains in which the operator $J$ given in \eqref{eq:J} becomes skew-adjoint. (Actually in \cite{gernandt} a slightly more involved version of $J$ is studied.)
}
\end{remark}
In analogy with \eqref{eq:Hamiltonianstaticshifted}, let
$z_*(t):= Xx_*(t), v_*(t):=Uu_*(t), r_*(t):=Pp_*(t)$, and define  the \emph{shifted} Hamiltonian functional
\bq
\label{eq:Hamiltonianinfshifted}
\begin{array}{rcl}
\widetilde{\cH}(z,v,r) &=&  \frac{1}{2} \int_0^T  (z (t,\tau)- z_*(t))^\top X^{-1} (z(t,\tau)- z_*(t))  
\\[2mm]
&&+ \, (v (t,\tau) - v_*(t))^\top U^{-1} (v (t,\tau) - v_*(t))\\[2mm]
&& +\, (r(t,\tau) - r_*(t))^\top P^{-1} (r(t,\tau) -r_*(t)) \, dt.
\end{array}
\eq
Furthermore, note that $(x_*,u_*,p_*)$ satisfies \eqref{eq:pont} (and thus corresponds to the solution of the optimal control problem) if and only if the right-hand side of \eqref{eq:pde} is zero for $(x_*,u_*,p_*)$ (i.e., $(x_*,u_*,p_*)$ is an equilibrium of \eqref{eq:pde} with respect to algorithmic time $\tau$). It follows that the incremental port-Hamiltonian system \eqref{eq:pde} can be rewritten as
\bq
\label{eq:primaldualinfshifted}
\begin{array}{rcl}
\frac{\partial}{\partial \tau} \bma z \\[2mm] v \\[2mm] r \ema & = &
\bma 
0 & 0 & - \frac{\partial}{\partial t} - A^\top \\[2mm]
0 & 0 & -B^\top \\[2mm]
- \frac{\partial}{\partial t} + A & B & 0
\ema
\bma  \frac{\delta \widetilde{\cH}}{\delta z} \\[2mm]  \frac{\delta \widetilde{\cH}}{\delta v}\\[2mm] \frac{\delta \widetilde{\cH}}{\delta r} \ema \, -
\\[10mm]
&&
\bma \frac{\partial K}{\partial x}(x,u) - \frac{\partial K}{\partial x}(x_*,u_*)\\[2mm] \frac{\partial K}{\partial u}(x,u) - \frac{\partial K}{\partial u}(x_*,u_*) \\[2mm] 0 \ema.
\end{array}
\eq
This constitutes an \emph{infinite-dimensional port-Hamiltonian system} as defined in \cite{JGP}. Compared to the interpretation of the examples in \cite{JGP} the 'physical time' $t$ replaces the spatial variable, while instead the 'algorithm time' $\tau$ becomes the time variable. Note that the formally skew-adjoint operator \eqref{eq:J} defines the lossless part of this infinite-dimensional port-Hamiltonian system, while the last part of \eqref{eq:primaldualinfshifted} corresponds to 'energy dissipation'. 

In analogy with \eqref{eq:staticdissipation}, we obtain by formal skew-adjointness of the operator \eqref{eq:J} the inequality
\bq
\label{eq:infdissipation}
\begin{array}{rcl}
\frac{d}{d\tau} \widetilde{\cH} & = &- \int_0^T (z(t,\tau)-z_*(t))^\top X^{-1} \\[2mm]
&& [\frac{\partial K}{\partial x}(x(t.\tau),u(t,\tau)) - \frac{\partial K}{\partial x}(x_*(t),u_*(t))] \, dt   \\[2mm]
& - &  \int_0^T (v(t,\tau)-v_*(t))^\top U^{-1} \\[2mm]
&& [\frac{\partial K}{\partial u}(x(t,\tau),u(t,\tau)) - \frac{\partial K}{\partial u}(x_*(t),u_*(t))] \, dt   
\\[2mm]
& - & (z(t,\tau) -z_*(t))^\top X^{-1}P^{-1} (r(t,\tau) -r_*(t))|_{t=0}^{t=T} \\[2mm]
& = & - \int_0^T (x(t,\tau)-x_*(t))^\top \\[2mm]
&& [\frac{\partial K}{\partial x}(x(t,\tau),u(t,\tau)) - \frac{\partial K}{\partial x}(x_*(t),u_*(t))]  \, dt   \\[2mm]
& - & \int_0^T (u(t,\tau)-u_*(t))^\top \\[2mm]
&& [\frac{\partial K}{\partial u}(x(t,\tau),u(t,\tau)) - \frac{\partial K}{\partial u}(x_*(t),u_*(t))]  \, dt \\[2mm]
& - & (x(t,\tau) - x_*(t))^\top  (p(t,\tau) -  p_*(t))|_{t=0}^{t=T}   \\[2mm]
& \leq & (x(0,\tau) - x_0)^\top  (p(0,\tau) -p_*(0)) \\[2mm]
& & -\, (x(T,\tau) - x_*(T))^\top  p(T,\tau),
\end{array}
\eq
by convexity of $K(x,u)$ and $x_*(0)=x_0, p_*(T)=0$. This suggests to impose the boundary conditions on \eqref{eq:primaldualinfshifted}
\bq
\label{eq:bound}
x(0,\tau)=x_0, \; p(T,\tau)=0,
\eq
%%are satisfied for all $\tau \geq 0$ then $\frac{d}{d\tau} \widetilde{\cH} \leq 0$ for all $\tau \geq 0$. 
in line with the mixed initial and final conditions of Pontryagin's Minimum principle \eqref{eq:pont}.

In order to assess the \emph{asymptotic stability} of the equilibrium $(z_*,v_*,r_*)$ of \eqref{eq:primaldualinfshifted}, and therefore $(x_*,u_*,p_*)$ of \eqref{eq:pde}, we perform the following LaSalle invariance principle analysis.
\begin{proposition}
\label{prop:lasalle}
Consider \eqref{eq:primaldualinfshifted} with boundary conditions \eqref{eq:bound}. Assume the pair $(A,B)$ is controllable. Then the largest invariant set contained in $\{ (z(t,\tau),v(t,\tau),r(t,\tau)), t \in [0,T] \mid \frac{d}{d\tau} \widetilde{\cH} = 0 \}$ is equal to the single function $(z_*(t),v_*(t),r_*(t)), t \in [0,T]$.
\end{proposition}
{\it Proof}
First of all, by strict convexity of $K(x,u)$
\bq
\begin{array}{l}
\{ (z(t,\tau),v(t,\tau),r(t,\tau)), t \in [0,T] \mid \frac{d}{d\tau} \widetilde{\cH} = 0 \}= \\[2mm] \{ (z(t,\tau),v(t,\tau),r(t,\tau)) \mid z(t,\tau)=z_*(t), v(t,\tau)=v_*(t), \\[2mm]
\quad t \in [0,T] \}.
\end{array}
\eq
Substitution of $z(t,\tau)=z_*(t), v(t,\tau)=v_*(t)$ into \eqref{eq:primaldualinfshifted} implies that $p(t,\tau)$ is independent of $\tau$, while this $p(t),t \in [0,T],$ satisfies
\bq
\left(\frac{d}{dt} + A^\top \right) (p(t) - p_*(t)) = 0, \; B^\top (p(t) - p_*(t)) =0.
\eq
By controllability of $(A,B)$ (and thus observability of $(B^\top,A^\top)$) this implies $p(t) = p_*(t), t \in [0,T]$.
${}_\blacksquare$

\smallskip

In analogy with Proposition \ref{prop} it is thus expected that under the conditions of Proposition \ref{prop:lasalle} and for a suitable choice of function spaces for $z(t,\tau),v(t,\tau),p(t,\tau), t \in [0,T], \tau \geq 0,$ and ensuing technical conditions, a (weak) version of LaSalle's invariance principle will show convergence to the optimal control solution $x_*(t),u_*(t),p_*(t), t \in [0,T],$ given by \eqref{eq:pont}.
%\begin{remark}
%{\rm
%Instead of \eqref{eq:bound} one could also consider the alternative boundary conditions
%\bq
%\begin{array}{rcll}
%\label{eq:boundcontrol}
%p(0,\tau) - p_*(0) & = & - D_0 (x(0,\tau) - x_*(0)), \quad & D_0=D_0^\top \geq 0, \; \tau\geq 0 \\[2mm]
%x(T,\tau) - x_*(T) & = & - D_T (p(T,\tau) - p_*(T)), \quad & D_T=D_T^\top \geq 0, \; \tau\geq 0
%\end{array}
%\eq
%Substituting into \eqref{eq:infdissipation} yields
%\bq
%\begin{array}{ll}
%\frac{d}{d\tau} \widetilde{\cH} \leq & -  (x(0,\tau) - x_*(0))^\top D_0 (x(0,\tau) - x_*(0)) \\[2mm]
%& - (p(T,\tau) - p_*(T))^\top D_T (p(T,\tau) - p_*(T)) \leq 0.
%\end{array}
%\eq
%Hence if $D_0>,D_T>0$ this implies the extra invariance conditions $x(0,
%\end{remark}
\begin{remark}
{\rm 
If the boundary conditions \eqref{eq:bound} are \emph{not} satisfied then \eqref{eq:primaldualinfshifted} amounts to a \emph{boundary control} port-Hamiltonian system as defined in \cite{JGP}, with boundary variables $x(0,\tau) - x_0,  p(0,\tau)-p_*(0), x(T,\tau) - x_*(T), p(T,\tau)$. In this case, as detailed in \cite{JGP}, the formally skew-adjoint operator $J$ defines an infinite-dimensional Dirac structure (called \emph{Stokes-Dirac structure}).
}
\end{remark}
\begin{remark}
{\rm 
In analogy with Remark \ref{rem} several generalizations are worth to be explored. An obvious one concerns the use of positive operators instead of the positive definite matrices $X,U,P$. Furthermore, it is of interest to include state and input constraints into the optimal control problem and the ensuing primal-dual gradient algorithm; see also \cite{lefevre}. Finally, similarly to \eqref{eq:primaldualstaticuy}, one may also augment the incremental port-Hamiltonian pde system \eqref{eq:pde} with \emph{inputs} and \emph{outputs}, for purposes of control by interconnection and distributed optimal control. 
}
\end{remark}

\section{Conclusions}
In analogy with the use of continuous time primal-dual gradient methods for static constrained convex optimization, this note has taken the basic steps in the explicit formulation and analysis of PDGCT for optimal control problems with linear dynamics and convex cost; taking inspiration from \cite{gernandt}. The resulting dynamics has been shown to be a system of port-Hamiltonian pde's, involving ordinary physical time and algorithmic time. This leads to several research questions, one being the design of \emph{sub-optimal} control strategies that are directly based on numerical schemes for the system of pde's \eqref{eq:pde}.

\end{document}